\documentclass[a4paper,12pt]{article}
\usepackage{}
\usepackage{mathrsfs}
\usepackage{amsfonts}
\usepackage{amscd,color}
\usepackage{amsmath,amsfonts,amssymb,amscd}
\usepackage{indentfirst,graphicx,epsfig}
\usepackage{graphicx,psfrag}
\input{epsf}

\newtheorem{df}{Definition}[section]
\newtheorem{thm}{Theorem}[section]

\newtheorem{lem}{Lemma}[section]

\newtheorem{pro}{Proposition}[section]
\newenvironment {proof} {\noindent{\em Proof.}}{\hspace*{\fill}$\Box$\par\vspace{4mm}}

\def\qed{\hfill \nopagebreak\rule{5pt}{8pt}}

\title{Conflict-free connection number of random graphs}

\author{
\small  Ran Gu$^1$,  Xueliang Li$^2$\\
\small $^1$College of Science, Hohai University,\\
\small Nanjing, Jiangsu 210098, China\\
\small $^2$Center for Combinatorics and LPMC\\
\small Nankai University, Tianjin 300071, China\\
\small Emails: rangu@hhu.edu.cn; lxl@nankai.edu.cn
\\
\date{}}
\begin{document}
\maketitle
\begin{abstract}
 An edge-colored graph $G$ is conflict-free connected if any two of its vertices are connected by a path which contains a color used on exactly one of its edges. The conflict-free connection number of a connected graph $G$, denoted by $cfc(G)$, is the smallest number of colors needed in order to make $G$ conflict-free connected.  In this paper, we show that almost all graphs have the conflict-free connection number 2. More precisely, let $G(n,p)$ denote the Erd\H{o}s-R\'{e}nyi random graph model, in which each of the $\binom{n}{2}$ pairs of vertices appears as an edge with probability $p$ independent from other pairs.  We prove that for sufficiently large $n$, $cfc(G(n,p))\le 2$ if $p\ge\frac{\log n +\alpha(n)}{n}$, where $\alpha(n)\rightarrow \infty$. This means that as soon as $G(n,p)$ becomes connected with high probability, $cfc(G(n,p))\le 2$. \\[2mm]
\end{abstract}

\section{Introduction}

All graphs in this paper are finite, simple and undirected. We follow \cite{BM} for graph theoretical notation and terminology not defined here. Let $G$ be a nontrivial connected graph with an {\it edge-coloring} $c : E(G)\rightarrow \{1, 2, \ldots, t\},
\ t \in \mathbb{N}$, where adjacent edges may have the same color. If adjacent edges of $G$ are assigned different colors by $c$, then $c$ is a proper (edge-)coloring. For a graph $G$, the minimum number of colors needed in a proper coloring of $G$ is referred to as the edge-chromatic number of $G$ and denoted by $\chi'(G)$. A path of $G$ is said to be a {\it rainbow path} if no two edges on the path have the same color. The graph $G$ is called {\it rainbow connected} if for any two vertices of $G$ there is a rainbow path of $G$ connecting them. An edge-coloring of a connected graph is a {\it rainbow connecting coloring} if it makes the graph rainbow connected. This concept of rainbow connection of graphs was introduced by Chartrand et al. \cite{CJMZ} in 2008. The \emph{rainbow connection number} $rc(G)$ of a connected graph $G$, is the smallest number of colors that are needed in order to make $G$ rainbow connected.   The interested readers can see \cite{LSS,LS, LS1} for surveys on this topic.

Motivated by rainbow coloring and proper coloring in graphs, Andrews et al. \cite{ALLZ} and Borozan et al. \cite{BFGMMMT} independently introduced the concept of proper-path coloring. Let $G$ be a nontrivial connected graph with an edge-coloring. A path in $G$ is called a \emph{proper path} if no two adjacent edges of the path are colored the same. An edge-coloring of a connected graph $G$ is a \emph{proper-path coloring} if every pair of distinct vertices of $G$ are connected by a proper path in $G$. The minimum number of colors that are needed to produce a proper-path coloring of $G$ is called the \emph{proper connection number} of $G$, denoted by $pc(G)$. From the definition, it follows that $1\le pc(G)\le min\{ rc(G), \chi'(G)\}\le m$, where $\chi'(G)$ is the chromatic index of $G$ and $m$ is the number of edges of $G$. For more details we refer to \cite{LM,LMQ}.

A coloring of the vertices of a hypergraph $\mathcal{H}$ is called \emph{conflicted-free} if each hyperedge $F$ of $\mathcal{H}$ has a vertex of unique color that is not repeated in $F$. The smallest
number of colors required for such a coloring is called the \emph{conflict-free chromatic number} of $\mathcal{H}$. This parameter was first introduced by Even et al. \cite{ELRS} in a geometric setting, in connection with frequency assignment problems for cellular networks. There are many results on the conflict-free coloring, see \cite{CKP,CT,PT}.

Motivated by the conflict-free colorings of hypergraphs and the rainbow and proper connections of graphs, Czap et al. \cite{CJV} introduced the concept of conflict-free connection for graphs. An edge-colored graph $G$ is called \emph{conflict-free connected} if each pair of distinct vertices is connected by a path which contains at least one color used on exactly one of its edges. This path is called a \emph{conflict-free path}, and this coloring is called a \emph{conflict-free connection coloring} of $G$. The \emph{conflict-free connection number} of a connected graph $G$, denoted by $cfc(G)$, is the smallest number of colors needed to color the edges of $G$ so that $G$ is conflict-free connected. It is easy to see that the parameter $cfc(G)$ has monotone property, i.e., for any connected spanning subgraph $G'$ of a graph $G$, one has $cfc(G)\le cfc(G')$. There are quite many results in the study of conflict-free connection of graphs, see \cite{CDHJLS, CHLMZ, CJV, DLLMZ, LZZMZJ}.

The study on rainbow connectivity of random graphs has attracted the interest of many researchers, see \cite{Caro, FT, HL}. In \cite{GLQ}, Gu at el. determined the proper connection number of random graphs. In this paper, we will focus on the parameter $cfc(G)$ for random graphs. The most frequently occurring probability model of random graphs is the Erd\H{o}s-R\'{e}nyi random graph model $G(n,p)$ \cite{ER}. The model $G(n,p)$ consists of all graphs with $n$ vertices in which the edges are chosen independently and with probability $p$. We say an event $\mathcal{A}$ happens \textit{with high  probability} if the probability that it happens approaches $1$ as $n\rightarrow \infty $, i.e., $Pr[\mathcal{A}]=1-o_n(1)$. Sometimes, we say \textit{w.h.p.} for short.
We will always assume that $n$ is the variable that tends to infinity.

Let $G$ and $H$ be two graphs on $n$ vertices. A property $P$ is said to be \emph{monotone} if whenever $G\subseteq H$ and $G$ satisfies $P$, then $H$ also satisfies $P$. For any property $P$ of graphs and any positive integer $n$, define $Prob(P, n)$ to be the ratio of the number of graphs with $n$ labeled vertices having the property $P$ divided by the total number of graphs with these vertices. If $Prob(P, n)$ approaches 1 as $n$ tends to infinity, then we say that almost all graphs have the property $P$. Similarly, for a fixed integer $r$, we say that almost all $r$-regular graphs have the property $P$ if the ratio of the number
of $r$-regular graphs with $n$ labeled vertices having property $P$ divided by the total number of $r$-regular graphs with these vertices approaches 1 as $n$ tends to infinity.

There are many results in the literature asserting that
almost all graphs have some property. Here we list some of them, which are related to our study on the conflict-free connection number of random graphs.

\begin{thm}\label{thac}\cite{BH}
For every $k \in \mathbb{N}$, almost all graphs are $k$-connected.
\end{thm}

\begin{thm}\label{thar}\cite{RW}
For fixed $r\ge 3$, almost all $r$-regular graphs are Hamiltonian.
\end{thm}

In \cite{CJV}, Czap et al. got the following result.
\begin{thm}\label{thd2}
If G is a noncomplete 2-connected graph, then $cfc(G)=2$.
\end{thm}
In \cite{CHLMZ}, the authors weaken the condition of the above theorem and got the following
result.

\begin{thm}\label{th0}\cite{CHLMZ,DLLMZ}
Let $G$ be a noncomplete 2-edge-connected graph. Then $cfc(G)=2$.
\end{thm}

From Theorem \ref{thac}, it is easy to get that almost all graphs are 2-connected. Hence, with Theorem \ref{thd2} or Theorem \ref{th0}, we have
\begin{thm}\label{thap}
Almost all graphs have the conflict-free connection number 2.
\end{thm}

Even if we concentrate on regular graphs, from Theorems \ref{thar} and \ref{thd2} or \ref{th0}, and the monotone property of $cfc(G)$, we also have the following result.
\begin{thm}
For fixed $r\ge 3$, almost all $r$-regular graphs have the conflict-free connection number 2.
\end{thm}

Let $C(G)$ denote the subgraph of a graph $G$ induced on the set of cut-edges of $G$. Recall that a linear forest is a forest where each of its components is a path. The following Theorem will be used in the sequel.

\begin{thm}\label{cjv}\cite{CJV} If $G$ is a connected graph, and $C(G)$ is a linear forest in which each component is of order 2, then $cfc(G) = 2$.
\end{thm}

In this paper, we mainly study the value of conflict-free connection number of random graph $G(n,p)$, when $p$ belongs to different ranges.
The following theorem is a classical result on the connectedness of a random graph.
\begin{thm}\label{ppt}\cite{ER}
Let $p=(\log n +a)/n$. Then
\begin{equation*}
Pr[G(n,p)\  is \ connected)]\rightarrow
\left\{
  \begin{array}{ll}
   e^{-e^{-a}} & \hbox{ if $|a|=O(1)$,} \\
    0 & \hbox{  $a\rightarrow -\infty$}, \\
    1 & \hbox{ $a\rightarrow +\infty$.}
  \end{array}
\right.
\end{equation*}
\end{thm}

Since the concept of conflict-free connection coloring only makes sense when the  graph is connected, we only study on the conflict-free connection coloring of $G(n,p)$ which is w.h.p. connected.
Our main result is as follows.
\begin{thm}\label{thm1}
For sufficiently large $n$, $cfc(G(n,p))\le2$ if $p\ge\frac{\log n +\alpha(n)}{n}$, where $\alpha(n)\rightarrow \infty$.
\end{thm}

Recall that, for a graph property $P$, a function $p(n)$ is called a \emph{threshold function} of $P$ if:
\begin{itemize}
\item for every $r(n) = \Omega(p(n))$, $G(n, r(n))$ w.h.p. satisfies $P$; and

\item for every $r'(n) = o(p(n))$, $G(n, r'(n))$ w.h.p.
does not satisfy $P$.
\end{itemize}
From Theorem \ref{ppt} and Theorem \ref{thm1}, we can obtain that the threshold for $cfc(G(n, p))=2$ is equal to the threshold for $G(n, p)$ to be connected. The proof of Theorem \ref{thm1} is given in Section 2.

\section{Proof of Theorem \ref{thm1}}

For a vertex subset $S$ of a graph $G$, we use $e(S)$ to denote the number of edges of  the subgraph $G[S]$ induced by $S$ of $G$. For two disjoint vertex subsets $X$ and $Y$ of $G$, let $E(X,Y)$ be the set of edges with one endpoint in $X$ and the other in $Y$, and $e(X,Y)=|E(X,Y)|$. For vertex subsets $U\subset S$, $N(U,S)$ is the disjoint neighbor set of $U$ in $G[S]$, i.e.  $N(U,S)=\{w\in S-U:\ \exists u\in S\ and \ \{uw\}\in G[S]\}$ and $d_S(v)=|N(v)\cap S|$ is the degree of $v$ in $S$.

From Theorem \ref{th0}, we know that if the random graph is 2-connected, then the conflict-free connection number is at most 2. In fact, it is known that \cite{B} if $p=\frac{1}{n}\{\log n+\log\log n+\omega(n)\}$, then w.h.p. $G(n,p)$ is  Hamiltonian, where $\omega(n)\rightarrow \infty$. Since Hamiltonian is a monotone property, we obtain that w.h.p. $cfc(G(n,p))=2$ if $\frac{1}{n}\{\log n+\log\log n+\omega(n)\}\le p<1 $.
Thus in the sequel, we assume that $p=\frac{\log n +\alpha(n)}{n}$, where $\alpha(n)=o(\log n)$, and $\alpha(n)\rightarrow \infty$.

For ease of notation, let $G\in G(n,p)$ and denote by $V$ the vertex set of $G(n,p)$.
We call a vertex $u$ \emph{large} if its degree $d(u)\geq \frac{\log n}{10}$ and \emph{small} otherwise. Let $V_1$ denote the vertex-subset consisting of all the small vertices, and $V_2$ be the vertex-subset consisting of all the large vertices. Namely, $V=V_1\cup V_2$. We first present some structure properties of $G(n,p)$, then we will give a conflict-free coloring of $G(n,p)$, which uses exactly 2 colors.
\subsection{Structure properties of $G(n,p)$}
\begin{pro}\label{prop1}
The followings hold w.h.p.\@ in $G$:
\begin{itemize}
  \item [(1)]For any $S \subseteq V$, $|S|\le \frac{n}{38}$ implies $|E(G[S])|< \frac{|S|np}{25}$.
  \item [(2)]If $U, W \subseteq V$, $U\cap W= \emptyset$, $|U|, |W| \ge \frac{n}{\log\log n}$, then $e(U, W)>0$.
\end{itemize}
\end{pro}

\begin{proof}
(1) The number of edges in an induced subgraph $G[S]$ with $|S|=s$ $(s\ge 2)$ is a binomial random variable with parameters $\binom{s}{2}$ and $p$. By Bollob\'{a}s \cite{B} (see page 14) we have that for large deviations of binomial random variables
$$\Pr\left[the \ number\  of\ edges\ in\ G[S]\ge\gamma \binom{s}{2}p \right]<\left(\frac{e}{\gamma}\right)^{\gamma \binom{s}{2}p}.$$ Setting $\gamma=\frac{2n}{25s},$ we obtain that
\begin{align*}
& \sum\limits_{s = 2}^{\frac{n}{{38}}}{\left( {\begin{array}{*{20}{c}}
n\\
s
\end{array}} \right)}\left(\frac{e}{\gamma}\right)^{\gamma \binom{s}{2}p}\le \sum\limits_{s = 2}^{\frac{n}{{38}}} {\left( \frac{ne}{s}\right)^s}\left( \frac{25es}{2n}\right)^{\frac{2n}{25s}\frac{s(s-1)}{2}\frac{\log n}{n}}\\
&= \sum\limits_{s = 2}^{\frac{n}{{38}}} {\left( \frac{ne}{s}\right)^s}\left( \frac{25es}{2n}\right)^{(s-1) \frac{\log n}{25}}\le\sum\limits_{s = 2}^{\frac{n}{{38}}} { e^{ -0.14(s-1){\log n}}}\\
&=\sum\limits_{s = 2}^{\frac{n}{{38}}} n^{-0.14(s-1)}=o(1),
\end{align*}
which implies that the statement of (1) w.h.p. holds.

(2) Let $\mathcal {A}$ denote the event that there exist two subsets $U, W \subseteq V$, $U\cap W= \emptyset$, $|U|, |W| \ge \frac{n}{\log\log n}$ and $e(U, W)=0$. Then
\begin{align*}
\Pr[\mathcal {A}]&\le \sum_{s \ge \frac{n}{\log \log n}}\sum_{t \ge \frac{n}{\log \log n}}\binom{n}{s}\binom{n-s}{t}(1-p)^{st}\\
&\le  \sum_{s \ge \frac{n}{\log \log n}}\sum_{t \ge \frac{n}{\log \log n}}{\left(\frac{ne}{s}\right)}^s{\left(\frac{ne}{t}\right)}^te^{-pst}\\
&\le \sum_{s \ge \frac{n}{\log \log n}}\sum_{t \ge \frac{n}{\log \log n}}{(ne)}^{s+t}{\left(\frac{1}{s}\right)}^s{\left(\frac{1}{t}\right)}^te^{-\frac{\log n}{n}\cdot \frac{n}{\log \log n}\cdot \frac{n}{\log \log n}}\\
&\le \sum_{s \ge \frac{n}{\log \log n}}\sum_{t \ge \frac{n}{\log \log n}}{(ne)}^{s+t}{\left(\frac{\log \log n}{n}\right)}^s{\left(\frac{\log \log n}{n}\right)}^te^{-\frac{n\log n}{{(\log \log n)}^2}}\\
&= \sum_{s \ge \frac{n}{\log \log n}}\sum_{t \ge \frac{n}{\log \log n}}{e}^{(s+t)\left(1+\log \log  \log n\right)}e^{-\frac{n\log n}{{(\log \log n)}^2}}\\
&\le \sum_{s \ge \frac{n}{\log \log n}}\sum_{t \ge \frac{n}{\log \log n}}{e}^{n\left(1+\log \log \log n\right)}e^{-\frac{n\log n}{{(\log \log n)}^2}}\\
&\le n^2{e}^{n\left(1+\log \log \log n\right)}e^{-\frac{n\log n}{{(\log \log n)}^2}}\le o(n^{-1}).
\end{align*}
\end{proof}

The following result focuses on the properties of small vertices in $G$.
\begin{pro}\label{prop2}
The followings hold w.h.p. in $G$.
\begin{itemize}
  \item [(1)]$|V_1|\le n^{0.4}$.
  \item [(2)]No pair of small vertices are adjacent or share a common neighbor.
  \item [(3)]There are at most $n^{0.5}$ edges incident with vertices in $V_1$.
\end{itemize}
\end{pro}
\begin{proof}
(1) Let $s=\lceil n^{0.4}\rceil$. Denote by $\mathcal{A}$ the event that there exists a vertex-subset $S$ with order $s$ such that every vertex $v\in S$ is small. Then $\mathcal{A}$ happens with probability
\begin{align*}
\Pr[\mathcal{A}]& \le\binom{n}{s}\left[\sum\limits_{k = 0}^{\frac{{\log n}}{{10}}} {\left( {\begin{array}{*{20}{c}}
n\\
k
\end{array}} \right)} {p^k}{(1 - p)^{n - 1 - k}}\right]^s\\
                & \le  {\left( {\frac{{ne}}{s}} \right)^s}{\left[ {\frac{{\log n}}{{10}}{{\left( {\frac{{10ne}}{{\log n}}} \right)}^{\frac{{\log n}}{{10}}}}{{\left( {\frac{{\log n + \alpha (n)}}{n}} \right)}^{\frac{{\log n}}{{10}}}}{e^{ - \frac{{\log n + \alpha (n)}}{n}\left( {n - 1 - \frac{{\log n}}{{10}}} \right)}}} \right]^s}  \\
                &\le  {\left( {\frac{{ne}}{s} \cdot \frac{{\log n}}{{10}}{{\left( {11e} \right)}^{\frac{{\log n}}{{10}}}}{e^{ - (\log n + \alpha (n) )+ \frac{{\log n + \alpha (n)}}{n} + \frac{{\log n}}{{10}} \cdot \frac{{\log n + \alpha (n)}}{n}}}} \right)^s}\\
                &\le{\left( {\frac{{ne}}{s} \cdot \frac{{\log n}}{{10}} \cdot {n^{\frac{3}{{10}}}} \cdot {n^{ - 1}} \cdot O(1)} \right)^s}\le O(n^{-0.05\cdot s}).
\end{align*}
That implies that w.h.p. $|V_1|\le n^{0.4}$.

(2) Let $\mathcal{B}$ denote the event that there exist two small vertices $x$, $y$ and the distance between $x$ and $y$ is at most 2.
We have
\begin{align*}
\Pr[\mathcal{B}]&\le \binom{n}{2}\Bigg\{p\left(\sum\limits_{i = 1}^{\frac{{\log n}}{{10}} } {\left( {\begin{array}{*{20}{c}}
{n - 2}\\
i
\end{array}} \right)} {p^i}{\left( {1 - p} \right)^{n - 2 - i}}\right)^2\\
                &\null+\binom{n-2}{1}p^2\left(\sum\limits_{i = 1}^{\frac{{\log n}}{{10}} } {\left( {\begin{array}{*{20}{c}}
{n - 3}\\
i
\end{array}} \right)} {p^i}{\left( {1 - p} \right)^{n - 3 - i}}\right)^2\Bigg\}\\
                                 &\le n^2\Bigg\{\frac{\log n +\alpha(n)}{n}\left[2\left(\frac{ne}
                                 {\frac{{\log n}}{{10}}}\right)^{\frac{{\log n}}{{10}}}p^{\frac{{\log n}}{{10}}}\left(1-p
                                 \right)^{n-2-\frac{{\log n}}{{10}}}\right]^2\\
                &\null+(n-2){\left(\frac{\log n +\alpha(n)}{n}\right)}^2\left[2\left(\frac{ne}
                                 {\frac{{\log n}}{{10}}}\right)^{\frac{{\log n}}{{10}}}p^{\frac{{\log n}}{{10}}}\left(1-p
                                 \right)^{n-2-\frac{{\log n}}{{10}}}\right]^2\Bigg\}\\
                                 &\le  n^2\left[\frac{\log n +\alpha(n)}{n}+n\left(\frac{\log n +\alpha(n)}{n}\right)^2\right]
                                 \left[2\binom{n}{\frac{{\log n}}{{10}}}p^{\frac{{\log n}}{{10}}}(1-p)^{n-2-\frac{{\log n}}{{10}}}\right]^2\\
                                 &\le\left[n(2\log n)+n(2\log n)^2\right]\left[2\left(\frac{ne}
                                 {\frac{{\log n}}{{10}}}\right)^{\frac{{\log n}}{{10}}}p^{\frac{{\log n}}{{10}}}\left(1-p
                                 \right)^{n-2-\frac{{\log n}}{{10}}}\right]^2\\
                                 &\le\left[n(2\log n)+n(2\log n)^2\right]\left[2\left(\frac{ne}
                                 {\frac{{\log n}}{{10}}}\right)^{\frac{{\log n}}{{10}}}\left(\frac{\log n }{n}\right)^{\frac{{\log n}}{{10}}}e
^{-\frac{\log n}{n}\left(n-\frac{{\log n}}{{10}}\right)}\right]^2\\
                                 &\le\left[n(2\log n)+n(2\log n)^2\right]n^{-1.3}\le n^{-0.2}.
\end{align*}
So w.h.p. no pair of small vertices are adjacent or share a common neighbor.

(3)  From (2), we get  that $V_1$ is w.h.p. an independent set, i.e., there is no edge in the induced subgraph $G[V_1]$. Since the degree of a small vertex is less than $\frac{\log n}{10}$, we have that the number of edges incident to $V_1$ is w.h.p. no more than
$$|V_1| \cdot \frac{\log n}{10} \le n^{0.4} \cdot \frac{\log n}{10}< n^{0.5}. $$
\end{proof}

Let $\mathscr{H}=\{G \in G(n,p)$: the conditions of Propositions \ref{prop1} and \ref{prop2} hold$\}$. From Proposition \ref{prop1} (2), we know that $V_1$ is an independent set in $G\in\mathscr{H}$. However, we claim that $G[V_2]$ is connected, in fact, $G[V_2]$ is Hamiltonian.  To prove this, we use the arguments similar to those of Cooper  et al. \cite{CF} and Frieze et al. \cite{FK}.
We regard the edges in $G$ as initially colored blue, but with the option of recoloring a set $R$ of the edges red. We require that the set $R$ of red edges is ``removable", which is defined as follows.

\begin{df}
\begin{itemize}
  \item [(1)] A set $R \subseteq E(G)$ is removable if
  \begin{itemize}
  \item [(i)]$R$ is a matching, and
  \item [(ii)]no edge of $R$ is incident with a small vertex, and
  \item [(iii)]$|R|=\lceil n^{0.4}\rceil$.
  \end{itemize}
  \item [(2)]Let $G_B[V_2]$ denote the subgraph of $G[V_2]$ induced by blue edges.
  \item [(3)]$N_B(U, V_2)$ denotes the disjoint neighbor set of $U$ in $G_B[V_2]$.
\end{itemize}
\end{df}

The following definitions and results are taken from P\'{o}sa \cite{P} and Frieze et al. \cite{FK}.
\begin{df}
Let $\Gamma=(V,E)$ be a non-Hamiltonian graph with a longest path of length $\ell$. A pair $\{u,v\}\notin E$ is called a hole if adding $\{u,v\}$ to $\Gamma$ creates a graph $\Gamma'$ which is Hamiltonian or contains a path longer than $\ell$.
\end{df}

\begin{df}
A graph $\Gamma=(V,E)$ is called a $(k,c)$-expander if $|N(U)|\ge c|U|$ for every subset $U \subseteq V(G)$ of cardinality $|U|\le k$.
\end{df}
Frieze et al. proved the relation between non-Hamiltonian $(k,c)$-expanders and holes, and their result is as follows.
\begin{lem}\cite{FK}\label{lem6}
Let $\Gamma$ be a non-Hamiltonian connected $(k,2)$-expander. Then $\Gamma$ has at least $\frac{k^2}{2}$ holes.
\end{lem}

Now we prove the following lemma, which is important to our proof of Theorem \ref{thm1}.
\begin{lem}\label{claim1}
The induced subgraph $G[V_2]$ of $G$ is w.h.p. Hamiltonian.
\end{lem}

\begin{proof}
Let $G \in \mathscr{H}$, for any sets $U$ and $S$, such that $U \subseteq S \subset V$ and $|U| \le \frac{n}{150}$. Let $F \subset E(G[S])$ and consider the graph $H=(S,F)$. If $U$ satisfies that the degree of $w$ in $H$ is at least $\frac{\log n}{11}$ for all $w \in U$, then, by Proposition \ref{prop1} (1), $|N(U, S)| \ge 3|U|$ in $H$. Moreover, by Proposition \ref{prop2} (2), each vertex $w \in U$ has at most one neighbor in $V_1$. We have $d_{V_2}(w) \ge \frac{\log n}{10}-1\ge \frac{\log n}{11}$.  Hence, we obtain that there are at least $3|U|$ neighbors of $U$ in $V_2$. Thus the removal of $\min\{|R|, |U|\}$ removable edges makes $|N_B(U, V_2)| \ge 2|U|$. So for $U \subseteq V_2$, $|U| \le \frac{n}{150}$, we have that $|N_B(U, V_2)| \ge 2|U|$. Hence, if $G[V_2]$ is not connected, then the smallest component cannot consist of less than $\frac{n}{150}$ vertices.
On the other hand, by Proposition \ref{prop1} (2), any two sets of vertices of size at least $\frac{n}{\log\log n}$ must be connected by an edge. So $G[V_2]$ is connected.

Since  for any $U \subseteq V_2$, $|U| \le \frac{n}{150}$, we have that $|N_B(U, V_2)| \ge 2|U|$, it is easy to get that $G[V_2]$ is a $(\frac{n}{150}, 2)$-expander. Combine with Lemma \ref{lem6}, $G[V_2]$ has at least $\frac{1}{2}(\frac{n}{150})^2$ holes depending only on $G_B[V_2]$. We define the set $\mathscr{B}$ to be those $G \in \mathscr{H}$ for which the subgraph $G[V_2]$ is not Hamiltonian. Our aim is to prove the following equation:
\begin{equation}\label{lem7}
\frac{|\mathscr{B}|}{|G(n,p)|}=o(1).
\end{equation}

To prove equation \eqref{lem7}, let $R$ be a set of red edges of $G$ and satisfying the property $P$ such that

(i) $R$ is removable, and

(ii) $\ell(G[V_2])=\ell(G_B[V_2])$,

where $\ell(H)$ is the length of a longest path in the graph $H$.

Let $\mathscr{C}$ be the set of all red-blue colorings of $\mathscr{B}$ which satisfy $P$. Let $\ell=\ell(G[V_2])$, we have $\ell < |V_2|$. Recall that by Proposition \ref{prop2} (3), there are at most $\mu=\lceil n^{0.5}\rceil$ edges incident with small vertices. Let $m$ be the number of edges in $G$, and $\Delta$ be the maximum degree of $G$. It is known that $\Delta$ is w.h.p. at most $3np$ (see e.g. \cite{B}). Set $r=|R|$. Since $R$ is a matching, we can choose it in at least
\begin{align*}
&\frac{1}{r!}(m-\ell-\mu)(m-\ell-\mu-2\Delta)
\ldots(m-\ell-\mu-2(r-1)\Delta)\\
&\ge \frac{1}{r!}(m-|V_2|-\mu)(m-|V_2|-\mu-2\Delta)
\ldots(m-|V_2|-\mu-2(r-1)\Delta)\\
&\ge \frac{(m-|V_2|)^r}{r!}(1-o(1))
\end{align*}
ways.
Hence, $$|\mathscr{C}|\ge |\mathscr{B}|\frac{(m-|V_2|)^r}{r!}(1-o(1)).$$

Consider that we fix the blue subgraph. Then, by the definition of holes, we have to avoid replacing at least $\frac{1}{2}(\frac{n}{150})^2$ edges when adding back the red edges in order to construct a red-blue coloring satisfying property $P$. Thus
$$|\mathscr{C}|\le \binom{\binom{n}{2}}{m-r}\binom{\binom{n}{2}
-(m-r)-\frac{1}{2}(\frac{n}{150})^2}{r}.$$
It follows that
\begin{align*}
\frac{|\mathscr{B}|}{|G(n,p)|}&\le \frac{\sum\limits_{m=\frac{1}{100}\binom{n}{2}p}^{\binom{n}{2}}
\left[\binom{\binom{n}{2}}{m-r}\binom{\binom{n}{2}-(m-r)
-\frac{1}{2}(\frac{n}{150})^2}{r}
\Bigg/\frac{(m-|V_2|)^r}{r!}(1-o(1))\right]}{\binom{\binom{n}{2}}{\binom{n}{2}p}}\\
\end{align*}
Note that $\frac{\left[\binom{\binom{n}{2}}{m-r}\binom{\binom{n}{2}-(m-r)
-\frac{1}{2}(\frac{n}{150})^2}{r}
\Bigg/\frac{(m-|V_2|)^r}{r!}(1-o(1))\right]}{\binom{\binom{n}{2}}{\binom{n}{2}p}}
\le O(e^{-\frac{r}{150^2}+\frac{nr}{(n-1)\log n}})
$ (see \cite{CF}) and $O(e^{-\frac{r}{150^2}+\frac{nr}{(n-1)\log n}})=o(n^{-\theta})$ for any constant $\theta>0$.
Thus, for any constant $\theta>3$, we have $$\frac{|\mathscr{B}|}{|G(n,p)|}\le
\sum\limits_{m=\frac{1}{100}\binom{n}{2}p}^{\binom{n}{2}}{o(n^{-\theta})}\le
n^2o(n^{-\theta})\le o(n^{-1}).$$
The proof is thus complete.
\end{proof}

\subsection{Color the edges of $G(n,p)$}

From Proposition \ref{prop2}, we can obtain that every small vertex is adjacent to a large vertex, and there is at most one small vertex among the neighbors of a large vertex. Thus, we can find a matching $M$ consisting of $|V_1|$ edges in $G$ such that for every edge $e$ in $M$, one endpoint of $e$ is small and the other endpoint is large. Let $s=|M|=|V_1|$. Denote the large vertices in $M$ by $u_1,u_2,\ldots,u_s$ and denote the small vertices in $M$ by $v_1,v_2,\ldots,v_s$. Without loss of generality, we assume that for every $i\in \{1,2,\cdots, s\}$, $\{u_iv_i\}$ is an edge in $M$. Denote the Hamiltonian cycle of $G[V_2]$ by $C$. Then the edge set $E(C)\bigcup E(M)$ induces a connected spanning subgraph $G'$ of $G$, and $C(G')$ satisfies the condition of Theorem \ref{cjv}. Therefore, by the monotone property one has that $cfc(G)\le
cfc(G')=2$.

To prove $cfc(G) \le 2$, one can also easily give $G$ an edge-coloring with 2 colors and verify that edge-coloring is conflict-free. Denote the Hamiltonian cycle of $G[V_2]$ by $C$. Let $e$ be an edge of $C$, we color $e$ with color 2, and the other edges in $C$ and $M$ with color 1.  It is easy to get that under this partial coloring, for every two distinct large vertices $x$ and $y$, there is a
$x-y$ path in $C$ containing the edge $e$. So this $x-y$ path  is conflict-free. For any small vertex $v_i$ and large vertex $z$, combining the edge $v_iu_i$ with the conflict-free $u_i-z$ path, we obtain a   conflict-free path connecting $v_i$ and $z$. Similarly, by using the edges in $M$ and the conflict-free path along $C$, we can find a conflict-free path connecting every pair of small vertices. Hence,  we can obtain that $cfc(G(n,p))\le 2$.
Thus Theorem \ref{thm1} follows. \qed

\vspace{3mm}

\noindent {\bf Acknowledgement.}  The first author is partially supported by
 Natural Science Foundation of Jiangsu Province (No. BK20170860), National Natural Science Foundation of China (No. 11701143), and Fundamental Research Funds for the Central Universities (No. 2016B14214). The second author is partially supported by NSFC No.11871034, 11531011 and NSFQH No.2017-ZJ-790.

\end{document}